\documentclass[a4paper]{amsart}
\usepackage[english]{babel}
\usepackage{amsmath, amssymb, amsthm, amscd}
\usepackage{enumerate}
\usepackage{paralist}
\usepackage[a4paper]{geometry}
\usepackage{url}
\usepackage{hyperref}
\usepackage[all]{xy}
\usepackage{mathscinet}

\usepackage{color}

\DeclareMathOperator{\im}{im}

\newcommand{\NN}{\mathbb{N}}
\newcommand{\RR}{\mathbb{R}}
\newcommand{\ZZ}{\mathbb{Z}}

\newcommand{\PM}{\mathcal{P}M}

\newcommand{\wgt}{\mathrm{wgt}}

\newcommand{\zcl}{\mathrm{zcl}}

\newcommand{\TC}{\mathsf{TC}}
\newcommand{\cat}{\mathsf{cat}}
\newcommand{\secat}{\mathsf{secat}}

\theoremstyle{plain}
\newtheorem{theorem}{Theorem}[section]
\newtheorem{prop}[theorem]{Proposition}
\newtheorem{lemma}[theorem]{Lemma}
\newtheorem{cor}[theorem]{Corollary}

\theoremstyle{definition}
\newtheorem{definition}[theorem]{Definition}

\theoremstyle{remark}

\newtheorem{example}[theorem]{Example}

\begin{document}
\setcounter{tocdepth}{1}

\title{Topological complexity of symplectic manifolds}
\author{Mark Grant}
\address{Institute of Pure and Applied Mathematics \\ University of Aberdeen \\ Fraser Noble Building  \\ Aberdeen AB24 3UE \\ United Kingdom}
\email{mark.grant@abdn.ac.uk}

\author{Stephan Mescher}
\address{Mathematisches Institut \\ Universit\"at Leipzig \\ Augustusplatz 10 \\ 04109 Leipzig \\ Germany}
\email{mescher@math.uni-leipzig.de}

\thanks{The authors wish to thank Ay\c se Borat, Michael Farber, Jarek K\c edra, and John Oprea for helpful comments regarding an earlier draft of this paper.}

\date{\today}

\begin{abstract}
We prove that the topological complexity of every symplectically atoroidal manifold is equal to twice its dimension. This is the analogue for topological complexity of a result of Rudyak and Oprea, who showed that the Lusternik--Schnirelmann category of a symplectically aspherical manifold equals its dimension. Symplectically hyperbolic manifolds are symplectically atoroidal, as are symplectically aspherical manifolds whose fundamental group does not contain free abelian subgroups of rank two. Thus we obtain many new calculations of topological complexity, including iterated surface bundles and symplectically aspherical manifolds with hyperbolic fundamental groups.

\end{abstract}

\maketitle

\section{Introduction}

Topological complexity is a numerical homotopy invariant, originally defined by M.\ Farber in \cite{FarberTC} and motivated by the motion planning problem from robotics. Let $X$ be a topological space, and let $PX=C^0([0,1],X)$ be the space of paths in $X$ endowed with the compact-open topology. Consider the \emph{free path fibration} $\pi: PX\to X \times X$ given by $\gamma \mapsto (\gamma(0),\gamma(1))$. The \emph{topological complexity of $X$}, denoted $\TC(X)$, is the minimal number $k$ (or infinity) for which there exists an open cover $\{U_0,U_1,\dots,U_k\}$ of $X \times X$ with the property that for each $j \in \{0,1,\dots,k\}$ there exists a continuous map $s_j: U_j \to PX$ with $\pi \circ s_j = \operatorname{incl}_j:U_j\hookrightarrow X\times X$. (Note that we use the \emph{reduced} topological complexity, which is one less than Farber's original definition.) Due to its potential applicability and intrinsic interest, this invariant has attracted a lot of attention from homotopy theorists in recent years. Overviews of its main properties and applications are given in \cite{FarberSurveyTC} and \cite[Chapter 4]{FarberBook}.

The invariant $\TC(X)$ is similar in spirit and properties to the classical \emph{Lusternik--Schnirelmann category} $\cat(X)$---the minimal number $k$ for which there exists an open cover $\{U_0, U_1, \ldots , U_k\}$ of $X$ such that each inclusion $U_j\hookrightarrow X$ is null-homotopic. (Here again we are using the \emph{reduced} version.) Both invariants possess \emph{upper bounds} in terms of dimension and connectivity. Namely, if $X$ is an $(r-1)$-connected CW complex, then
$$ \cat(X) \leq \dim X/r \qquad \text{and} \qquad \TC(X) \leq 2 \dim X/r.$$
Proofs of these bounds may be found in \cite[Theorem 1.50]{CorneaLS} for $\cat(X)$, and in \cite[Theorem 5.2]{FarberTC2} for $\TC(X)$.

Both invariants also admit elementary \emph{lower bounds} in terms of cup-length in cohomology.
Recall that the \emph{cup-length} of $X$, denoted $\operatorname{cl}(X)$, is the minimal $k$ such that any product of $k+1$ elements of $\widetilde{H}^*(X)$ vanishes. The \emph{zero-divisors cup-length} of $X$, denoted $\operatorname{zcl}(X)$, is the minimal $k$ such that any product of $k+1$ elements in the kernel of the homomorphism $\Delta^*: H^*(X\times X)\to H^*(X)$ induced by the diagonal vanishes. Then we have
$$ \cat(X) \geq \operatorname{cl}(X) \qquad \text{and} \qquad \TC(X) \geq \zcl(X).$$
There are analogous statements with local coefficients. These statements, which may be found in \cite[Proposition 1.5]{CorneaLS} and \cite[Theorem 7]{FarberTC} respectively, are both specializations of \cite[Theorem 4]{SchwarzGenus}. Indeed, both $\cat(X)$ and $\TC(X)$ can be seen as special cases of the notion of \emph{Schwarz genus} (or \emph{sectional category}) of fibrations.

The goal of this article is to give new computations of the topological complexity of symplectic manifolds. Our results will apply more generally to the class of \emph{cohomologically symplectic}, or \emph{$c$-symplectic manifolds} \cite{LuptonOpreaCsymp}. Our convention is that a \emph{$c$-symplectic manifold} is a pair $(M,\omega)$ consisting of a closed manifold $M$ of even dimension $2n$, together with a closed $2$-form $\omega\in \Omega^2(M)$ such that the $n$-th power $[\omega]^n\in H^{2n}(M;\RR)$ of the cohomology class represented by $\omega$ is nonzero. (Choosing a representative of the $c$-symplectic class gives a slight refinement of the definition used in \cite{LuptonOpreaCsymp}, which is unnecessary but simplifies some statements.) In particular, a closed symplectic manifold $(M,\omega)$ is $c$-symplectic. If $M$ is a simply connected $c$-symplectic manifold of dimension $2n$, then the bounds given above (together with a simple cohomology calculation in the case of $\TC$, carried out in \cite{FTY}) imply that
$$\cat(M)=n\qquad\text{and}\qquad \TC(M)=2n.$$

Our interest therefore lies in the non-simply connected case. Here, one can often use the notion of \emph{category weight} of cohomology classes, introduced by E.\ Fadell and S.\ Husseini \cite{FadellHusseini} and later refined by Y.\ Rudyak \cite{RudyakWeight} and J.\ Strom \cite{StromThesis}, in order to improve on the cup-length lower bound for $\cat(M)$. Recall that a ($c$-)symplectic manifold $(M,\omega)$ is called \emph{($c$-)symplectically aspherical} if
$$\int_{S^2} f^*\omega = 0$$
for all smooth maps $f: S^2 \to M$. Using category weight, Rudyak and J.\ Oprea derived the following result, a key tool in Rudyak's proof of the Arnold Conjecture for symplectically aspherical manifolds.

\begin{theorem}[{\cite[Corollary 4.2]{RudyakOprea}}]
\label{ThmRO}
Let $(M,\omega)$ be a $c$-symplectically aspherical manifold. Then $\cat(M) = \dim M$. 	
\end{theorem}

Hence for $c$-symplectically aspherical manifolds, the dimensional upper bound for $\cat(M)$ is attained. The analogous statement for topological complexity (that $\TC(M)=2 \dim M$ when $M$ is $c$-symplectically aspherical) does not hold. For example, the torus $T^2=S^1 \times S^1$ equipped with its standard volume form is aspherical, therefore symplectically aspherical. However, as shown in \cite{FarberTC}, we have $\TC(T^2)=2 < 4=2\dim T^2$.

Our main result identifies an extra condition which ensures that the topological complexity of a $c$-symplectically aspherical manifold achieves its dimensional upper bound, thereby giving an analogue of Theorem \ref{ThmRO} for topological complexity. We say that a ($c$-)symplectic manifold $(M,\omega)$ is \emph{($c$-)symplectically atoroidal} if
$$	\int_{T^2} f^*\omega = 0$$
	for all smooth maps $f:T^2\to M$.

\begin{theorem}
\label{ThmAtor}
	Let $(M,\omega)$ be a $c$-symplectically atoroidal manifold. Then $\TC(M)=2\dim M$.
\end{theorem}

 Note that since there is a degree-one map $T^2 \to S^2$, every $c$-symplectically atoroidal manifold is $c$-symplectically aspherical. The torus itself is symplectically aspherical but not symplectically atoroidal, as are many symplectic nilmanifolds or solvmanifolds. J.\ K\c edra has shown that every symplectically hyperbolic manifold is symplectically atoroidal \cite{KedraHyp}, while A.\ Borat has shown that every symplectically aspherical manifold whose fundamental group does not contain subgroups isomorphic to $\ZZ\oplus\ZZ$ is symplectically atoroidal \cite{Borat}. Thus our result applies to many examples, including iterated bundles of higher genus surfaces, symplectic manifolds of negative sectional curvature, and symplectically aspherical manifolds with hyperbolic fundamental group. We mention also that M.\ Brunnbauer and D.\ Kotschick \cite{BrunnKot} have given an example of a closed $4$-manifold which is symplectically atoroidal without being symplectically hyperbolic.

Symplectic atoroidality may be viewed as imposing a condition of asymmetry on a symplectic manifold, in the following sense (we are indebted to John Oprea for the statement and its proof).

\begin{prop}
A symplectically atoroidal manifold $(M,\omega)$ does not admit any non-trivial symplectic $S^1$-actions.
\end{prop}
\begin{proof}
Let $(M,\omega)$ be symplectically atoroidal. It follows immediately from the definitions that the \emph{flux homomorphism} $F_\omega :\pi_1(\operatorname{Symp}_0(M,\omega))\to H^1(M;\RR)$ of \cite{Calabi,Banyaga} (see also \cite{Borat}) is zero. Then by \cite[Lemma 3.12]{LuptonOpreaCsymp} and the remarks preceding, any symplectic $S^1$-action on $(M,\omega)$ is $c$-Hamiltonian. By \cite[Corollary 3.7]{LuptonOpreaCsymp}, any $c$-Hamiltonian $S^1$-action on $(M,\omega)$ has a fixed point. However, by \cite[Theorem 4.16]{LuptonOpreaCsymp}, any non-trivial $S^1$-action on a $c$-symplectically aspherical manifold is fixed point free. Combining these results gives the Proposition.
\end{proof}

Thus our results tend to support the idea, explored in \cite{Grant}, that symmetries of manifolds are responsible for lowering their topological complexity. Note however that D.\ Cohen and L.\ Vandembroucq have recently shown \cite{CohenVan} that the Klein bottle $K$ has $\TC(K)=2 \dim(K)$, despite admitting an effective $S^1$-action.

The proof of Theorem \ref{ThmAtor} employs the notion of \emph{$\TC$-weight} of cohomology classes. This analogue of category weight was introduced and studied by M.\ Farber and the first author in \cite{FarberGrantSymm} and \cite{FarberGrantWeights}, in order to improve on the zero-divisors cup-length lower bound for topological complexity. A class $u\in H^*(X\times X)$ has $\wgt(u)\ge 1 $ if and only if it is a zero-divisor, and has $\wgt(u)\ge 2$ if and only if it is in the kernel of the homomorphism induced by the \emph{fibrewise join} $p_2: P_2 X\to X\times X$ of the free path fibration with itself. We show that for a closed atoroidal form $\omega\in \Omega^2(M)$ the associated zero-divisor $[\bar\omega] = 1\times [\omega] - [\omega]\times 1 \in H^2(M\times M;\RR)$ has $\wgt([\bar\omega])=2$, by analyzing the Mayer--Vietoris sequence associated to the fibrewise join $P_2M$.

Our methods use de Rham theory on infinite-dimensional manifolds of smooth paths and loops, for which we adopt the formalism of Kriegl and Michor \cite{KrieglMichor}. Special care needs to be taken in passing between de Rham and singular cohomology, because the space $P_2M$ is not a manifold.

The paper is organized as follows. In Section \ref{weight}, we recall the definition of the  $\TC$-weight of cohomology classes and its most important properties before relating it to fiberwise joins and their Mayer-Vietoris sequence. Section \ref{MVmanifold} looks at this Mayer-Vietoris sequence from a different angle in the case of manifolds, where we derive an explicit description of the connecting homomorphism in terms of differential forms. This description is applied to $c$-symplectic manifolds in Section \ref{symp} and yields a proof of Theorem \ref{ThmAtor}. We apply this theorem in Section \ref{appl} to derive several new computations of the topological complexity of $c$-symplectic manifolds.

\section{The $\TC$-weight of a cohomology class and fiberwise joins}
\label{weight}

Throughout this section we fix a topological space $X$. We let $PX=C^0([0,1],X)$ and $LX = C^0(S^1,X)$ denote the continuous path and loop spaces, respectively. For a local coefficient system $A$ on $X \times X$, we will call a cohomology class $u \in H^*(X \times X;A)$ a \emph{zero-divisor} if
$$u \in \ker \left[\Delta^*:H^*(X\times X;A) \to H^*(X;\Delta^*A) \right] , $$
where $\Delta:X \to X \times X$ denotes the diagonal map. 

In the following, we will briefly recall the notion of $\TC$-weight from \cite{FarberGrantSymm} and \cite{FarberGrantWeights}. It is an analogue of the strict category weight defined by Y. Rudyak in \cite{RudyakWeight}.

\begin{definition}
Let $A$ be a local coefficient system on $X \times X$ and $k \in \NN_0$. The \emph{$\TC$-weight} of a class $u \in H^*(X \times X;A)$, denoted $\wgt(u)$, is defined to be the maximal $k$ such that $f^*u=0 \in H^*(Y;f^*A)$ for every continuous map $f: Y \to X \times X$ for which there exist an open cover $U_1\cup \dots\cup U_k=Y$ and continuous maps $f_i:U_i \to PX$ with $\pi \circ f_i = f|_{U_i}$ for $i \in \{1,2,\dots,k\}$.
\end{definition}

The following properties of $\TC$-weight, proven by M. Farber and the first author in \cite{FarberGrantSymm}, illustrate its importance for obtaining estimates from below of topological complexity.

\begin{theorem}
\label{ThmTCwgt}
Let $A$ and $B$ be local coefficient systems on $X \times X$, and let $u\in H^*(X \times X;A)$ and $v\in H^*(X\times X;B)$ be cohomology classes.
\begin{enumerate}
\item
For the cup product $u\cup v\in H^*(X\times X;A\otimes B)$ we have
\begin{equation*}
\wgt(u \cup v) \geq \wgt(u)+\wgt(v) .
\end{equation*}
\item If $u$ is nonzero and $\wgt(u) = k$, then $\TC(X) \geq k$.
    \end{enumerate}
\end{theorem}

A class in $H^*(X\times X;A)$ has positive $\TC$-weight if and only if it is a zero-divisor, see \cite[p. 3361]{FarberGrantWeights}. In the remainder of this section, we want to determine a criterion to decide if a given cohomology class has $\TC$-weight at least $2$. A key role is played by the \emph{fiberwise join} of the fibration $\pi: PX \to X \times X$ with itself, which we will denote by $p_2:P_2X \to X \times X$, see e.g. \cite[Section 2]{FGKV} for a detailed construction. More precisely, a result of Schwarz from \cite{SchwarzGenus} applied to the fibration $\pi$ shows the following:
\begin{prop}
\label{wgt2}
Let $A$ be a local coefficient system on $X \times X$, and let $u \in H^*(X \times X;A)$. If
$$u \in \ker \left[p_2^*: H^*(X \times X;A) \to H^*(P_2X;p_2^*A) \right]  , $$
then $\wgt(u) \geq 2$.
\end{prop}

We want to take a closer look at the cohomology of $P_2X$. We define continuous maps $r_1,r_2: LX \to PX$ by
\begin{equation}
\label{r1r2}
(r_1(\gamma))(t)= \gamma\left(\tfrac{t}2\right) , \quad (r_2(\gamma))(t)= \gamma\left(1-\tfrac{t}2 \right) \quad \forall \gamma \in LX, \ \ t \in [0,1] .
\end{equation}
One checks without difficulties that $(r_1,r_2): LX \to PX \times PX$ maps $LX$ homeomorphically onto the pullback of $PX \stackrel{\pi}{\rightarrow} X \times X \stackrel{\pi}{\leftarrow} PX$, so that the following is a pullback diagram:
\begin{equation}
\label{pullback}
\xymatrix{
LX \ar[r]^-{r_1} \ar[d]_{r_2} &  PX \ar[d]^\pi\\
PX \ar[r]^-\pi &  X \times X .
}
\end{equation}
The total space of the fibration $P_2X$ is a homotopy pushout of the diagram $PX \stackrel{r_1}\leftarrow LX \stackrel{r_2}\rightarrow PX$, and we let $i_1,i_2: PX \to P_2X$ and $p_2: P_2X \to X \times X$ be the induced maps making the following diagram commutative up to homotopy:
\begin{equation}
\label{pushout}
 \xymatrix{
LX \ar[r]^{r_1} \ar[d]_{r_2}& PX \ar[d]^{i_2} \ar@/^/[ddr]^{\pi} & \\
PX \ar[r]^{i_1} \ar@/_/[rrd]_{\pi} & P_2X \ar[dr]^{p_2} & \\
& & X \times X.
}
\end{equation}
In fact, by taking $P_2 X$ to be the double mapping cylinder of the maps $r_1$ and $r_2$, and taking $p_2$ to be the whisker map induced by the constant homotopy across the diagram (\ref{pullback}), we can arrange that the triangles in the above diagram strictly commute.

We want to consider the Mayer-Vietoris cohomology sequence associated with this homotopy pushout, see \cite[Chapter 21]{StromClassical}, which is of the form
$$\dots\to H^{k}(P_2X;A) \stackrel{i_1^*\oplus i_2^*}\to H^{k}(PX;A) \oplus H^{k}(PX;A) \stackrel{r_1^*-r_2^*}\to H^{k}(LX;A) \stackrel{\delta}\to H^{k+1}(P_2X;A) \to \dots   $$
for any abelian group $A$.

\begin{prop}
\label{PropMVimage}
Let $A$ be an abelian group and $k \in \NN$. A class $u \in H^k(X \times X;A)$ is a zero-divisor if and only if
$$p_2^*u \in \im \left[\delta:H^{k-1}(LX;A) \to H^k(P_2X;A) \right],$$
where $\delta$ denotes the connecting homomorphism of the above Mayer-Vietoris sequence.
\end{prop}
\begin{proof}
We consider the map $e:PX \to X$, $\gamma \mapsto \gamma(0)$, given by evaluating a path at its initial point. For each $\nu\in\{1,2\}$, the diagram
\begin{equation}
\label{ev}
\xymatrix{
PX \ar[r]^{i_\nu} \ar[d]_e & P_2 X \ar[d]^{p_2} \\
X \ar[r]^-\Delta &  X \times X .
}
\end{equation}
 commutes up to homotopy. Since the vertical map $e$ is well-known to be a homotopy equivalence, it follows that
$$\ker \left[\Delta^*:H^k(X \times X;A) \to H^k(X;A) \right]= \ker \left[i_\nu^*\circ p_2^*:H^k(X \times X;A) \to H^k(PX;A) \right].$$
Thus, $u \in H^k(X\times X;A)$ is a zero-divisor if and only if $$p_2^*u \in \ker \left[i_\nu^*: H^k(P_2X;A)\to H^k(PX;A) \right]$$
for $\nu \in \{1,2\}$ and the exactness of the Mayer-Vietoris sequence shows the claim.
\end{proof}

We next combine the previous propositions to obtain a criterion for a cohomology class to have $\TC$-weight two.

\begin{cor}
Let $A$ be an abelian group and $k \in \NN$ with $k \geq 2$. If the connecting homomorphism from the above Mayer-Vietoris sequence $\delta:H^{k-1}(LX;A) \to H^k(P_2X;A)$  vanishes, then every zero-divisor $u \in H^k(X \times X;A)$ satisfies $\wgt(u)\geq 2$.
\end{cor}
\begin{proof}
This is an immediate consequence of Propositions \ref{wgt2} and \ref{PropMVimage}.
\end{proof}

\section{The connecting homomorphism for manifolds}
\label{MVmanifold}

We want to make use of Proposition \ref{PropMVimage} to investigate zero-divisors that are induced by degree-two cohomology classes of a manifold. More precisely, given a manifold $M$ and a closed $2$-form $\omega \in \Omega^2(M)$ we want to determine a $1$-cocycle on $LM$ whose cohomology class maps to the class of $p_2^*[1\times \omega - \omega \times 1]$ under the connecting homomorphism
$$\delta: H^1(LM;\RR) \to H^2(P_2M;\RR) . $$
To work in a de Rham-theoretic setting, we briefly discuss differential forms and de Rham cohomology on smooth path and loop spaces of manifolds.

Let $M$ be a finite-dimensional smooth manifold. We consider the spaces
$$\Lambda M=C^\infty(S^1,M)\quad \text{and} \quad \PM=C^\infty([0,1],M), $$
both equipped with the respective Whitney topology. As a consequence of the results of \cite{KrieglMichor}, see also \cite{Stacey}, both $\Lambda M$ and $\PM$ possess the structure of infinite-dimensional Fr\'{e}chet manifolds, locally modelled on the Fr\'{e}chet spaces $C^\infty(S^1,\RR^{\dim M})$ and $C^\infty([0,1],\RR^{\dim M})$, resp. There are homotopy equivalences $\Lambda M\simeq LM$ and $\PM \simeq PM$, see \cite[Section 4.2]{Stacey} for details. We may thus consider $\Lambda M$ and $\PM$ as ``differentiable replacements" of spaces of continuous paths and loops. Moreover, it follows from the results of \cite[Section 42]{KrieglMichor} that with respect to this manifold structure, the evaluation maps $e_t: \Lambda M \to M$ and $e_t: \PM \to M$, both given by $x \mapsto x(t)$, are smooth for every $t \in [0,1]$. This particularly implies the smoothness of the endpoint evaluation map $\pi: \PM \to M \times M$, $\pi(\gamma)=(\gamma(0),\gamma(1))$.

Recall that $\TC(M)=\secat(\pi:PM\to M\times M)$, the sectional category of the free path fibration. The following lemma allows us to work with smooth paths throughout.

\begin{lemma}\label{lem:smoothTC}
If $M$ is a smooth manifold, then $\TC(M)=\secat(\pi:\PM\to M\times M)$.
\end{lemma}

\begin{proof}
The inclusion of smooth paths into continuous paths gives a commuting diagram
\[
\xymatrix{
\PM \ar@{^{(}->}[rr] \ar[dr]_{\pi} & & PM \ar[dl]^{\pi} \\
   & M\times M &
   }
   \]
   from which it is clear that $\TC(M)\leq \secat(\pi:\PM\to M\times M)$.

   So suppose $U\subseteq M\times M$ is an open set with a local section $s:U\to PM$ of the (continuous) path fibration. The adjoint of $s$ is a homotopy $H:U\times I\to M$ between the restricted projections ${\rm pr}_1|_{U}$ and ${\rm pr}_2|_{U}$. This is homotopic rel $U\times \partial I$ to a smooth homotopy $\tilde{H}:U\times I\to M$, see \cite[Theorem 6.29]{Lee} for example. The adjoint of $\tilde{H}$ gives a local section $\tilde{s}:U\to \PM$ of the smooth path fibration. Hence $\secat(\pi:\PM\to M\times M)\leq \TC(M)$, and we are done.
    \end{proof}

With $\Gamma$ denoting the space of smooth sections of a vector bundle, the tangent spaces of elements of $\Lambda M$ and $\PM$ are given by
 $$T_x\Lambda M = \Gamma(x^*TM)\quad \forall x  \in \Lambda M, \quad T_y\PM = \Gamma(y^*TM) \quad \forall y \in \PM  . $$
Tangent vectors at a loop or path are vector fields along that loop or path. The derivative $De_t:T_x\Lambda M\to T_{x(t)} M$ is given by evaluating the vector field at $x(t)$, and similarly for $De_t:T_y\PM\to T_{y(t)}M$. In particular, the map $\pi:\PM\to M\times M$ is a submersion, and the diagram (\ref{pullback}) remains a pullback if we replace $PM$ and $LM$ by $\PM$ and $\Lambda M$.
 	
For $k \in \NN$ we let $\Omega^k(M)$ denote the space of smooth real-valued $k$-forms on $M$. For each $\omega \in \Omega^k(M)$ we let $\bar\omega \in \Omega^k(M \times M)$ denote the form given by
$$\bar\omega = 1 \times \omega - \omega \times 1 :={\rm pr}_2^*\omega - {\rm pr}_1^*\omega  .$$
It is easy to see that if $\omega$ is closed, then $\bar\omega$ will be closed as well. Moreover, since ${\rm pr}_1\circ\Delta= {\rm pr}_2\circ\Delta$, the class $[\bar\omega] \in H^k(M \times M;\RR)$ will be a zero-divisor for every closed $\omega \in \Omega^k(M)$.

In analogy with the finite-dimensional case, we define the space of smooth $k$-forms on $\PM$ for each $k \in \NN$ by $\Omega^k(\PM) = \Gamma(L^k_a(T\PM))$, where $L^k_a(T\PM)$ denotes the bundle of alternating $k$-linear forms on $T\PM$. We define $\Omega^k(\Lambda M)$ similarly.

\begin{definition}\label{aomega}
Given $\omega \in \Omega^2(M)$ we let $\beta_\omega \in \Omega^1(\PM)$ be defined by
$$(\beta_\omega)_x[\xi] = \int_0^1 \omega_{x(t)}\big(\dot{x}(t),\xi(t)\big) \ dt \quad \forall x \in \PM, \ \ \xi \in T_x \PM  .$$
\end{definition}

\begin{prop}
\label{PropPathform}
If $\omega \in \Omega^2(M)$ is closed, then
$$\pi^*\bar\omega = d\beta_\omega \in \Omega^2(\PM)  , $$
where $\pi: \PM \to M \times M$ again denotes the endpoint evaluation map.
\end{prop}
\begin{proof}
We first note that
\[
\pi^*\bar\omega = \pi^*({\rm pr}_2^* - {\rm pr}_1^*)\omega = e_1^*\omega - e_0^*\omega \in \Omega^2(\PM).
\]
Given a smooth non-decreasing function $\phi:\RR\to[0,1]$ such that $\phi(t)=0$ for $t\leq 1/4$ and $\phi(t)=1$ for $t\geq 3/4$, the map
$$
e:\PM\times\RR\to M,\qquad e(x,t)= x(\phi(t))
$$
defines a smooth homotopy from $e_0$ to $e_1$. As in the proof of homotopy invariance of de Rham cohomology (given as \cite[Lemma 34.2]{KrieglMichor} in the infinite dimensional case) we can therefore define a homotopy operator $K:\Omega^*(M)\to\Omega^{*-1}(\PM)$ satisfying $e_1^*-e_0^*=dK + Kd$, as follows. For each $t\in \RR$, let ${\rm ins}_t:\PM\to \PM\times \RR$ be given by ${\rm ins}_t(x)=(x,t)$. Define an integral operator $I^1_0:\Omega^*(\PM\times\RR)\to \Omega^*(\PM)$ by
\[
I^1_0(\varphi) = \int^1_0 {\rm ins}_t^*\varphi \,dt.
\]
Let $T:=\left( 0,\partial/{\partial t}\right)$ be the unit vector field in the $\RR$ direction on $\PM\times \RR$, and denote by $\iota_T:\Omega^*(\PM\times\RR)\to \Omega^{*-1}(\PM\times \RR)$ the interior product with $T$. Finally, let $e^*:\Omega^*(M)\to \Omega^*(\PM\times \RR)$ be pullback along $e$.

Then the composition $K=I^1_0\circ \iota_T\circ e^*$ defines a homotopy operator, as shown in the proof of \cite[Lemma 34.2]{KrieglMichor}. If $\omega\in \Omega^2(M)$ is closed, it follows that
\[
 \pi^*\bar\omega = e_1^*\omega - e_0^*\omega = dK(\omega).
 \]

 It only remains to check that with the above definitions we have $K(\omega) = \beta_\omega\in \Omega^1(\PM)$. Given $x \in \PM$ and $\xi \in T_x \PM$, note that for every $t\in [0,1]$ we have $D({\rm ins}_t)_x(\xi) = (\xi,0)\in T_{(x,t)}(\PM\times\RR)\cong T_x\PM\times T_t\RR$, while $De_{(x,t)}(T)=\dot x(t)$ and $De_{(x,t)}(\xi,0) = \xi(t)$ (compare \cite[Corollary 42.18]{KrieglMichor}). It follows that
  \begin{align*}
 K(\omega)_x[\xi] & =\int^1_0  ({\rm ins}_t^* \iota_T e^*\omega)_x(\xi) \, dt \\
                  & = \int^1_0 (\iota_T e^*\omega)_{(x,t)}(\xi,0) \, dt \\
                  & = \int^1_0 \omega_{x(t)}\big(\dot x(t),\xi(t)\big)\, dt \\
                  & = (\beta_\omega)_x[\xi],
 \end{align*}
 as claimed.
\end{proof}

\begin{prop}
\label{PropLoopform}
Given a closed $2$-form $\omega \in \Omega^2(M)$ we let $\alpha_{\omega} \in \Omega^1(\Lambda M)$ be given by
$$\alpha_\omega = r_1^*\beta_\omega - r_2^*\beta_\omega, $$
where $\beta_\omega$ is given in Definition \ref{aomega} and $r_1,r_2: \Lambda M\to \PM$ are defined by equations (\ref{r1r2}). Then $d \alpha_\omega = 0$.
\end{prop}
\begin{proof}
Using Proposition \ref{PropPathform}, we have
\begin{align*}
d \alpha_\omega & = d(r_1^*\beta_\omega - r_2^*\beta_\omega) \\
                & = r_1^*d\beta_\omega - r_2^*d\beta_\omega \\
                & = r_1^*(e_1^*\omega - e_0^*\omega) - r_2^*(e_1^*\omega - e_0^*\omega) \\
                & = (r_2^*e_0^*\omega - r_1^*e_0^*\omega) + (r_1^*e_1^*\omega - r_2^*e_1^*\omega) \\
                & = 0,
\end{align*}
where the last equality holds since $e_0\circ r_1 = e_0 \circ r_2$ and $e_1\circ r_1 = e_1 \circ r_2$.
\end{proof}

It is a consequence of \cite[Theorem 34.7]{KrieglMichor} and \cite[Proposition 42.3]{KrieglMichor} that the de Rham cohomology groups of $\PM$ and $\Lambda M$ are well-defined, and that the de Rham Theorem holds for these manifolds. Indeed, for any smoothly paracompact smooth manifold $M$ (finite or infinite dimensional) the composition of cochain maps
\begin{equation}
\label{Psi12}
\Psi_M: \Omega^*(M) \to C^*_{\rm smooth}(M;\RR) \to C^*(M;\RR),
\end{equation}
induces an isomorphism $H^*_{\rm dR}(M;\RR)\cong H^*(M;\RR)$ from de Rham cohomology to singular cohomology with real coefficients. Here the first map is given by integrating over smooth simplices, and the second is induced by a smoothing operator $C_*(M)\to C_*^{\rm smooth}(M)$ on singular chains (see \cite[Chapter 18]{Lee}). It is not difficult to check using Stokes' Theorem that these cochain maps are natural, in the sense that given a smooth map $f:M\to N$ we have $f^*\circ \Psi_N = \Psi_M\circ f^*$.

\begin{theorem}
Let $\omega \in \Omega^2(M)$ be closed and let $\alpha_\omega\in \Omega^1(\Lambda M)$  be given as in Proposition \ref{PropLoopform}. Let
$$\delta: H^1(\Lambda M;\RR) \to H^2(P_2M;\RR)$$ denote the connecting homomorphism of the Mayer-Vietoris sequence from above for real coefficients. Then
$$\delta([\alpha_\omega])=p_2^*[\bar\omega] , $$
where we identify the de Rham cohomology classes $[\alpha_\omega]$ and $[\bar\omega]$ with their corresponding singular cohomology classes under the above-mentioned canonical isomorphisms.
\end{theorem}
\begin{proof}
By definition, the connecting homomorphism is obtained using the snake lemma in the first two rows of the following diagram:
$$
\xymatrix{
0 \ar[r] & C^1(P_2M;\RR) \ar[r]^{i_1^*\oplus i_2^*\qquad \quad} \ar[d]^{\partial}& C^1(\PM;\RR) \oplus C^1(\PM;\RR)\ar[d]^{\partial\oplus \partial}\ar[r]^{\qquad \quad r_1^*-r_2^*} & C^1(\Lambda M;\RR)\ar[d]^{\partial} \ar[r] & 0 \\
0 \ar[r] & C^2(P_2M;\RR) \ar[r]^{i_1^*\oplus i_2^*\qquad \quad} & C^2(\PM;\RR) \oplus C^2(\PM;\RR) \ar[r]^{\qquad \quad r_1^*-r_2^*} & C^2(\Lambda M;\RR) \ar[r] & 0 \\
 & C^2(M \times M;\RR) \ar[u]^{p_2^*} \ar[ur]_{\pi^* \oplus \pi^*}& & &
}
$$
where $C^i(-;\RR)$ denotes singular cochains with real coefficients and $\partial$ denotes the respective singular codifferential. The (strict) commutativity of the bottom triangle follows from the remark made immediately after diagram (\ref{pushout}).

To find a cocycle $a \in C^1(\Lambda M;\RR)$ whose cohomology class is mapped to the class of $p_2^*c$ by the connecting homomorphism for a given cocycle $c \in C^2(M \times M;\RR)$, it thus suffices to find such an $a$ for which there exist $b_1, b_2 \in C^1(\PM;\RR)$ such that $a = r_1^*b_1 - r_2^*b_2$ and
$\partial b_1 = \partial b_2 = \pi^*c$.

Passing to differential forms, we have seen in Propositions \ref{PropPathform} and \ref{PropLoopform} that $d\beta_\omega = \pi^*
\bar\omega$ and $\alpha_\omega = r_1^*\beta_\omega - r_2^*\beta_\omega$.  The result now follows from the naturality of the de Rham cochain equivalences $\Psi_{M\times M}$, $\Psi_{\PM}$ and $\Psi_{\Lambda M}$.
\end{proof}

\begin{cor}
\label{Corwgt2}
Let $\omega \in \Omega^2(M)$ be closed and let $\alpha_\omega\in \Omega^1(\Lambda M)$  be given as in Proposition \ref{PropLoopform}. If $\alpha_\omega$ is exact then $\wgt([\bar\omega]) \geq 2$.
\end{cor}

\section{The weight of the cohomology class of a $c$-symplectic form}
\label{symp}

The considerations for closed $2$-forms from the previous section are of particular interest for $c$-symplectic manifolds. Non-vanishing of the $n$-th power of the $c$-symplectic class is crucial in the proof of Theorem \ref{ThmAtor}.

\begin{prop}
\label{PropSymp}
Let $(M,\omega)$ be a $c$-symplectic manifold. If $[\alpha_\omega]=0 \in H^1(\Lambda M;\RR)$, then $\TC(M)=2 \dim M$.
\end{prop}
\begin{proof}
Let $2n = \dim M$. As before let $\bar\omega=1\times \omega - \omega \times 1 \in \Omega^2(M\times M)$, and let $[\bar\omega]=1\times[\omega]-[\omega]\times 1 \in H^2(M\times M;\RR)$ denote its cohomology class. By assumption, $[\omega]^n\in H^{2n}(M;\RR)$ is nonzero. It follows that $[\bar\omega]^{2n}\neq 0$, since it contains the nontrivial summand
$$(-1)^n \binom{2n}{n} [\omega]^n \times [\omega]^n.$$
Thus, using Theorem \ref{ThmTCwgt} (1) and Corollary \ref{Corwgt2} we derive that
$$\wgt\left([\bar\omega]^{2n} \right) \geq 2n \cdot \wgt([\bar\omega]) \geq 4n . $$
Finally, Theorem \ref{ThmTCwgt} (2) implies that $\TC(M) \geq 4n$. The opposite inequality $\TC(M) \leq 4n$ follows from the dimensional upper bound from the introduction.
\end{proof}

Combining Proposition \ref{PropSymp} and the atoroidality condition gives Theorem \ref{ThmAtor}.

\begin{proof}[Proof of Theorem \ref{ThmAtor}]
In light of Proposition \ref{PropSymp} it only remains to show that the condition of $\omega\in \Omega^2(M)$ being atoroidal implies that $[\alpha_\omega]=0 \in H^1(\Lambda M;\RR)$. For this it suffices to check that integrating $\alpha_\omega$ over a smooth $1$-cycle gives zero. Hence it suffices to show that
\[
\int_{S^1} c^*\alpha_\omega=0\qquad\mbox{for all }c \in C^\infty(S^1,\Lambda M).
 \]
 Given such $c$, we define $\tilde{c}: T^2 \to M$ by $\tilde{c}(s,t)= (c(s))(t)$ for all $s, t \in S^1$ and derive from \cite[Theorem 42.14]{KrieglMichor} that $\tilde{c}$ is smooth. Using Fubini's theorem we compute that
$$ \int_{S^1} c^*\alpha_\omega = \int_0^1 (\alpha_\omega)_{c(s)}[\dot{c}(s)] \ ds= \int_0^1 \int_0^1 \omega_{\tilde{c}(s,t)}(\partial_t\tilde{c}(s,t),\partial_s\tilde{c}(s,t)) \ dt \ ds = \int_{T^2} \tilde{c}^*\omega = 0 , $$
since $(M,\omega)$ is symplectically atoroidal. Since $c$ was chosen arbitrarily, it follows that $[\alpha_\omega]=0$ and Proposition \ref{PropSymp} yields the claim.
\end{proof}

\section{Applications}
\label{appl}

In this section we give some general settings in which our main result Theorem \ref{ThmAtor} applies.

\begin{theorem}
\label{ThmHypFund}
Let $(M,\omega)$ be a $c$-symplectically aspherical manifold, such that $\pi_1(M)$ does not contain $\ZZ\oplus\ZZ$ as a subgroup. Then $\TC(M)=2 \dim M$.
\end{theorem}

\begin{proof}
As shown in the proof of \cite[Theorem 4]{Borat}, if $f:T^2\to M$ is a smooth map such that $\int_{T^2} f^*\omega \neq 0$, and $\omega$ is aspherical, then $f_*:\pi_1(T^2)\to \pi_1(M)$ is injective. Hence the given conditions imply that $(M,\omega)$ is $c$-symplectically atoroidal, so that Theorem \ref{ThmAtor} applies.
\end{proof}

Note that this result improves by one the bound given in \cite[Theorem 8]{FarberMescher}. It applies in particular to the case of hyperbolic fundamental groups.

\begin{cor}
Let $(M,\omega)$ be a $c$-symplectically aspherical manifold, such that $\pi_1(M)$ is hyperbolic. Then $\TC(M)=2 \dim M$.
\end{cor}

\begin{cor}
Let $(M,\omega)$ be a $c$-symplectic manifold which admits a Riemannian metric of negative sectional curvature. Then $\TC(M)=2 \dim M$.
\end{cor}
\begin{proof}
By the Theorems of Cartan--Hadamard and Preissman, $M$ is aspherical and $\pi_1(M)$ does not contain any subgroups isomorphic to $\ZZ\oplus\ZZ$. Thus Theorem \ref{ThmHypFund} applies.
\end{proof}

We next discuss a related geometric condition which implies symplectic atoroidality.

\begin{definition}
Let $(M,\omega)$ be a symplectic manifold and let $p:\tilde{M}\to M$ be its universal cover. The pair $(M,\omega)$ is called \emph{symplectically hyperbolic} if there exists $\theta \in \Omega^1(\tilde{M})$, such that $p^*\omega = d\theta$ and such that
$$\sup_{q \in \tilde{M}} \|\theta_q \|_q < +\infty, $$
where $\| \cdot \|_q$ denotes the norm on $T^*_q\tilde{M}$ that is induced by the lift of a chosen Riemannian metric on $M$.
\end{definition}

This notion was introduced by L. Polterovich in \cite{Polt}, generalizing the concept of K\"ahler hyperbolicity introduced by M. Gromov in \cite{GromovHyper}. It has further been discussed by J. K\c{e}dra in \cite{KedraHyp} and by G. Paternain under the name of \emph{weakly exact forms with bounded primitive} in \cite{PaterMagnetic}. The simplest examples of symplectically hyperbolic manifolds are oriented surfaces of genus at least two together with their volume forms. The connection to our results is given by the following observation.

\begin{lemma}[{\cite[Proposition 1.9]{KedraHyp}, \cite[Lemma 2.3]{MerryCharge}}]
\label{LemmaHypAtor}
Symplectically hyperbolic manifolds are symplectically atoroidal.
\end{lemma}

Hence, Theorem \ref{ThmAtor} applies to all closed symplectically hyperbolic manifolds.

\begin{cor}
Let $\Sigma_g\to E\to M$ be an oriented surface bundle with fibre of genus $g\ge2$ over a closed symplectically hyperbolic manifold $(M,\omega)$. Then $\TC(E)=2 \dim(E)$.

In particular, if $E$ is the total space of an iterated sequence of oriented surface bundles, where the fiber of each iteration step and the base space of the first fibration are oriented surfaces of genus at least two, then $\TC(E)=2\dim E$.
\end{cor}
\begin{proof}
This follows from \cite[Corollary 2.2]{KedraHyp}, combined with Lemma \ref{LemmaHypAtor} and Theorem \ref{ThmAtor}.
\end{proof}

This is a far-reaching generalisation of the fact, easily obtained using zero-divisors cup-length estimates, that a finite product $\Sigma_{g_1}\times\cdots\times\Sigma_{g_n}$ of orientable surfaces of higher genus has
$$\TC(\Sigma_{g_1}\times \dots \times \Sigma_{g_n}) = 4n . $$

Finally, we give examples of $c$-symplectic but not symplectic manifolds to which our Theorem applies.

\begin{prop}\label{PropDegreeOne}
Let $(M,\omega)$ be a $c$-symplectically atoroidal manifold (with its induced orientation), and let $g:N\to M$ be a map of nonzero degree. Then the pair $(N,g^*\omega)$ is $c$-symplectically atoroidal.
\end{prop}
\begin{proof}
Since $g$ induces an isomorphism on top cohomology, it follows that the pair $(N,g^*\omega)$ is a $c$-symplectic manifold. If $f:T^2\to N$ is a smooth map, then
\[
\int_{T^2}f^*g^*\omega = \int_{T^2}(g\circ f)^*\omega = 0,
\]
since $\omega$ is atoroidal. Hence $g^*\omega$ is atoroidal.
\end{proof}

\begin{example}
Let $(M,\omega)$ be a $c$-symplectically atoroidal manifold, and let $g:\tilde M\to M$ be a branched covering of nonzero degree. Then $(\tilde M,g^*\omega)$ is $c$-symplectically atoroidal, and so $\TC(\tilde M)=2 \dim \tilde M$. Note that $g^*\omega$ fails to be symplectic over the branch locus. We do not know if $\tilde M$ necessarily carries a symplectic structure.
\end{example}

\begin{example}
Let $(M,\omega)$ be a closed symplectically atoroidal $4$-manifold, and let $X$ be any almost complex $4$-manifold (with the induced orientation). As shown in \cite[Proposition 1.3.1]{Audin}, the connected sum $M\# X$ does not possess any almost complex structure, and hence cannot possess any symplectic structure. However, there is a degree-one map $g:M\# X\to M$ collapsing $X$ to a point, which by Proposition \ref{PropDegreeOne} implies that $M\# X$ is $c$-symplectically atoroidal. Hence $\TC(M\# X)=8$.
\end{example}

 \bibliography{TC}
 \bibliographystyle{amsalpha}

\end{document}